\documentclass{amsart}   

\usepackage{amssymb}
\usepackage{amsmath}
\usepackage{amsthm}

\newtheorem{theorem}{Theorem}
\newtheorem{lemma}[theorem]{Lemma}
\newtheorem{cor}[theorem]{Corollary}
\newtheorem{prop}[theorem]{Proposition}

\newtheorem{exam}[theorem]{Example}

\newcommand{\R}{\mathbb{R}}
\newcommand{\N}{\mathbb{N}}
\newcommand{\Q}{\mathbb{Q}}

\newcommand{\K}{\mathcal{K}}
\newcommand{\B}{\mathcal{B}}
\newcommand{\D}{\mathcal{D}}

\newcommand{\M}{\mathcal{M}}

\newcommand{\omom}{\omega^\omega}

\newcommand{\tr}{\ge_T}
\newcommand{\te}{=_T}
\newcommand{\tq}{\ge_T}

\newcommand{\ctm}{\mathfrak{c}}

\newcommand{\cl}[1]{\overline{#1}}

\newcommand{\add}[1]{\mathop{\mathrm{add}}(#1)}
\newcommand{\cof}[1]{\mathop{\mathrm{cof}}(#1)}
\newcommand{\spec}[1]{\mathop{\mathrm{spec}}(#1)}

\newcommand{\h}{\text{h}}

\newcommand{\Pk}[1]{\left([#1]^{<\omega}\right)^{\omega}}

\newcommand{\pcf}[1]{\mathop{\mathrm{pcf}}(#1)}

\title{Tukey Order, Calibres and the Rationals}

\author{Paul Gartside}
\address{Department
    of Mathematics, University of Pittsburgh, Pittsburgh, PA~15260, USA}
\email{gartside@math.pitt.edu} 

\author{Ana Mamatelashvili}
\address{Department of Mathematics and Statistics, Auburn University, Auburn, AL~36849, USA}
\email{azm0105@auburn.edu}



\begin{document}

\begin{abstract}
One partially ordered set, $Q$, is a \emph{Tukey quotient} of another, $P$ -- denoted $P \geq_T Q$ -- if there is a map  $\phi : P \to Q$ carrying cofinal sets of $P$ to cofinal sets of $Q$.  Let $X$ be a space and denote by $\mathcal{K}(X)$ the set of compact subsets of $X$, ordered by inclusion. For certain separable metrizable spaces $M$, Tukey upper and lower bounds of $\mathcal{K}(M)$ are calculated. Results on invariants of $\mathcal{K}(M)$'s are deduced. The structure of all $\mathcal{K}(M)$'s under $\le_T$ is investigated. Particular emphasis is placed on the position of $\mathcal{K}(M)$ when $M$ is:  completely metrizable, the rationals $\mathbb{Q}$, co-analytic or analytic.  
\end{abstract}

\maketitle

\section{Introduction}
Given two directed sets $P$ and $Q$, we say $Q$ is a  \emph{Tukey quotient} of $P$, denoted $P \tq Q$,  if there is a map  $\phi : P \to Q$ carrying cofinal sets of $P$ to cofinal sets of $Q$. If $P$ and $Q$ are mutual Tukey quotients then they are said to be \emph{Tukey equivalent}, denoted $P \te Q$.
Introduced to study the notion of Moore-Smith convergence in topology \cite{MooreSmith,Tukey}, the Tukey order, $\tq$, on directed sets has become an important tool in the study of directed sets in general, and especially certain directed sets naturally arising in analysis and topology.

Key to the utility of the Tukey order is that it focusses on what happens cofinally in directed sets, and so is sufficiently coarse to allow comparison of very different directed sets, but nevertheless preserves many order invariants. 
Indeed if $P \tq Q$ then the cofinality (minimal size of a cofinal set) of $Q$ is no more than the cofinality of $P$, and the additivity (minimal size of an unbounded set) of $P$ is no more than the additivity of $Q$.
Similarly, if $P \tq Q$ and $P$ has calibre $(\kappa,\lambda,\mu)$ (every $\kappa$-sized subset contains a $\lambda$-sized subset all of whose $\mu$-sized subsets are bounded) then so does $Q$.
Fremlin was  the first to realize the relevance of the Tukey order in analysis, showing that a result of Bartoszynski \cite{Bart} and Raisonnier and Stern \cite{RS} on additivity of the measure and category ideals was due to a Tukey quotient between the relevant ideals. Then he started to systematically investigate the Tukey order on certain ideals arising in measure theory \cite{Fr3}.  Many of his ideals are families of compact subsets of a separable metrizable space, for example $\mathcal{E}_\mu$, the ideal of compact null subsets of $[0,1]$ and ${\small \text{NWD}}$, the ideal of compact nowhere dense subsets of $[0,1]$. Subsequently, in \cite{Frem}, Fremlin started to study the Tukey structure of $\K(M)$, the set of \emph{all} compact subsets of a separable metric space $M$. 
In this paper we develop further this line of investigation.

Let $\mathcal{M}$ be all separable metrizable spaces (up to homeomorphism), and $\K(\mathcal{M})$ all the Tukey equivalence classes of $\K(M)$ for $M$ in $\M$. 
Christensen had shown \cite{Chris} that if $M$ is separable metrizable and $\omega^\omega \tq \K(M)$ then $M$ is Polish. Fremlin observed that the initial structure of $\K(\M)$ under the Tukey order is as follows. At the bottom are $\K(M) \te 1$ corresponding to $M$ compact. As its unique immediate successor there is the class of $\K(M) \te \omega$ corresponding to $M$ which are locally compact, but not compact. Its unique successor is the class $\K(M) \te \omega^\omega$ corresponding to $M$ Polish but not locally compact. All other $\K(M)$ are strictly above $\omega^\omega$.
Fremlin further investigated how the Tukey order on $\K(\M)$ was related to the Borel and projective hierarchies. Recall that a separable metrizable space is \emph{absolutely Borel} if it is a Borel subset of some  metric compactification, \emph{analytic} (or, $\mathbf{\Sigma}^1_1$) if it is the continuous image of a Polish space; \emph{co-analytic} (or, $\mathbf{\Pi}^1_1$)  if the complement in a compact metric space of an analytic set; and $\mathbf{\Sigma}^1_2$ if the continuous image of a co-analytic set. 
Then it turns out that 
$\K(M) \te \K(\Q)$ whenever $M$ is co-analytic (in particular, when Borel) but not Polish, and if $\K(M) \le_T \K(\Q)$ then $M$  is $\mathbf{\Sigma}^1_2$. 
However (in ZFC) there is an analytic set $M$ such that $\K(\Q) <_T \K(M)$.  
Under Projective Determinacy, indeed, if $\K(M) \le_T \K(\Q)$ then $M$ is co-analytic, and  for \emph{every}  analytic non-Borel $M$ we have $\K(\Q) <_T \K(M)$. Fremlin also showed that the cofinality of $\K(\omega^\omega)$ and $\K(\Q)$ are $\mathfrak{d}$, the cofinality of $\omega^\omega$.

In \cite{KM} the authors explored the large scale structure of $\K(\M)$ under the Tukey order. We showed, among other things, that there is a $2^\ctm$-sized antichain in $\K(\M)$, and  $[0,1]^\omega$, with the co-ordinatewise order, order embeds in $\K(\M)$. Applications were given to function spaces and to certain compact spaces (Gul'ko compacta)  arising in functional analysis. (The authors have also, see \cite{KS}, examined the Tukey structure of $\K(S)$'s where $S$ is a subspace of $\omega_1$.)

Here our results come in two parts. In the first part (Section~\ref{Invs}) we investigate invariants of $\K(M)$ for certain `small' $M$ related to $\Q$. A separable metrizable space is \emph{totally imperfect} if  it contains no Cantor set, or equivalently every compact subset is countable. The rationals, $\Q$, are totally imperfect, as is any separable metrizable space of size strictly less than the continuum. The members of the $2^\ctm$-sized antichain of \cite{KM} are all totally imperfect. 
We give upper and lower Tukey bounds on $\K(M)$ for $M$ totally imperfect (Theorem~\ref{TI_bounds}). We deduce bounds on the cofinality of $\K(M)$'s, for $M$ totally imperfect; calculating the cofinality exactly for $M$ of size $< \aleph_\omega$; and give examples showing how Shelah's {\small PCF} theory intervenes at size $\aleph_\omega$ and higher. 
We then turn to calibres of $\K(M)$'s, for $M$ in $\M$. As is now well known, every such  $\K(M)$  has calibre $(\omega_1,\omega)$. For a fixed separable metric $M$  we concentrate on the problem of determining the cardinals $\kappa$ such that $\K(M)$ has calibre $\kappa$. It turns out to be more convenient to focus on the set -- the \emph{spectrum} of $\K(M)$ -- of all regular $\kappa$ which are \emph{not} calibres of $\K(M)$. Using the upper and lower Tukey bounds again, we show that for totally imperfect $M$ the problem of computing $\spec{\K(M)}$ rests on determining $\spec{\omega^\omega}$. And this is also strongly related to  {\small PCF} theory (see Theorem~\ref{specomom}, and prior discussion).   

In the second part of the paper (Section~\ref{Struct}) we investigate extensions of Christensen's theorem ($\omega^\omega \tq \K(M)$ implies $M$ Polish) to metrizable spaces of uncountable weight (the \emph{weight} of a space is the minimal size of a basis), and  continue Fremlin's examination of the initial structure of $\K(\M)$ under the Tukey order. Indeed in Section~\ref{compmet} we establish the outlines of the initial structure of $\K(M)$'s where $M$ is metrizable of some fixed uncountable weight $\kappa$. It follows that there is no possibility of a Christensen type theorem for uncountable weights: for no uncountable $\kappa$ is there a directed set $P_\kappa$ such that $P_\kappa \tq \K(M)$, where $M$ is metrizable of weight $\kappa$, forces $M$ to be completely metrizable. 
Next we give an internal description of those separable metrizable $M$ such that $\K(M) \le_T \K(\Q)$. The first author, Medini and Zdomskyy \cite{GaMedZdom} have used this to give an example, under $\mathbb{V}=\mathbb{L}$, of an analytic set $M$ such that $\omega^\omega \te \K(\omega^\omega) <_T \K(M) <_T \K(\Q)$. 
We also  determine those metrizable spaces $M$ such that $\K(\Q) \le_T \K(M)$,  answering a question of Fremlin. Finally we show that it is consistent that there is an analytic set $M$ such that $\K(M)$ is Tukey incomparable with $\K(\Q)$. The first author, Medini and Zdomskyy, building on our results here, have shown \cite{GaMedZdom} that it is also consistent that there are analytic, non-Borel, sets $M_1$ and $M_2$ such that $\K(M_1) <_T \K(\Q)$ and $\K(M_2) \te \K(\Q)$.   

\section{Tukey Preliminaries} 
The basic objects of study here are \emph{directed sets} -- partially ordered sets such that every finite subset has an upper bound. 
We write $\bigvee S$ for the least upper bound of a subset $S$ of a directed set $P$, if this exists. A directed set is \emph{Dedekind complete} if every bounded subset has a least upper bound.
Products of directed sets are given the (co-ordinatewise) product order. 
Products and powers, such as $\omega^\omega$, play an important role in our work. The mod-finite order, $\le^*$, on $\omega^\omega$  (so, $\sigma \le^* \tau$ if for some $N$ and all $n \ge N$ we have $\sigma (n) \le \tau (n)$) will also play a role.
But for us the most important class of directed sets are those of the form $\K(X)$, where $X$ is a topological space and $\K(X)$ is the family of compact subsets of $X$, ordered by set inclusion. Every $\K(X)$ is Dedekind complete.

A useful feature of $\K(X)$ is that carries a natural topology, the Vietoris topology, which interacts nicely with the order. A basic open set of $\K(X)$ in the Vietoris topology has the form $\langle U_1, \ldots , U_n\rangle = \{K \in \K(X) : K \subseteq \bigcup U_i \ \land \ K \cap U_i \ne \emptyset \text{ for } i =1, \ldots n\}$, where $U_1, \ldots, U_n$ are open subsets of $X$.
If $X$ is metrizable then the Hausdorff metric induces the Vietoris topology on $\K(X)$.

If $P'$ is a subset of a directed set $P$ then a subset $C$ of $P$ is \emph{cofinal} for $P'$ in $P$ provided for every $p'$ in $P'$ there is a $c$ in $C$ such that $c \ge p'$. (When $P'=P$ this corresponds to the usual definition of a cofinal set in $P$.) 
Define two invariants of a directed set by $\add{P}$ is the minimal size of an unbounded subset of $P$, and $\cof{P}$ is the minimal size of a cofinal set of $P$. 
Define, as usual, $\mathfrak{b}=\add{\omega^\omega, \le^*}$ (note, $\add{\omega^\omega}=\omega$) and $\mathfrak{d}=\cof{\omega^\omega}=\cof{\omega^\omega,\le^*}$.

We now introduce a generalization of the Tukey order on directed sets. Proofs of all claims can be found in \cite{KM}. 
Let $P'$ and $Q'$ be  subsets of directed sets $P$ and $Q$, respectively. A function $\phi : P \to Q$ is a (relative) \emph{Tukey quotient} if $\phi$ maps subsets of $P$ cofinal for $P'$ to subsets of $Q$ cofinal for $Q'$. We denote the existence of a relative Tukey quotient from $(P',P)$ to $(Q',Q)$ by $(P',P) \tq (Q',Q)$, and say \emph{$(P',P)$ Tukey quotients to $(Q',Q)$}. When $P'=P$ we abbreviate to $P \tq (Q',Q)$; and similarly for $Q'=Q$.

If $Q$ is Dedekind complete then $(P',P) \tq (Q',Q)$ if and only if there is a map $\phi : P \to Q$ such that $\phi$ is order preserving and $\phi (P)$ is cofinal for $Q'$ in $Q$. Since all directed sets mentioned here are Dedekind complete we almost always assume Tukey quotients are order preserving. 
There is also a dual version of Tukey quotients. We only use this in the absolute case (when $P'=P$ and $Q'=Q$) so we state it  in that generality. A function $\psi : Q \to P$ is a \emph{Tukey map} if $\psi$ maps unbounded sets in $Q$ to unbounded sets in $P$. Then there is a Tukey map from $Q$ to $P$ if and only if there is a Tukey quotient from $P$ to $Q$.
Naturally we write $Q \le_T P$ if $P \tq Q$, and $P \te Q$ -- $P$ and $Q$ are \emph{Tukey equivalent} -- if $P \tq Q$ and $P \le_T Q$. Clearly Tukey equivalence is an equivalence relation on the class of all directed sets, and $\le_T$ is a partial order on directed sets, up to Tukey equivalence.

We record for later use some basic Tukey properties of $\K(X)$. A map $f: X \to Y$ is \emph{compact-covering} if it is continuous and for every compact subset $L$ of $Y$ there is a compact subset $K$ of $X$ such that $f(K) \supseteq L$. 
\begin{lemma}\label{km_pres}
Let $X$  be a space. Then 
(1) if $A$ is a closed subset of $X$ then $\K(X) \tq \K(A)$ and (2) if $Y$ is the compact-covering image of $X$ then $\K(X) \tq \K(Y)$.

Also, (3) if $\{X_\lambda : \lambda \in \Lambda\}$ is a family of spaces then $\K(\prod_\lambda X_\lambda) \te \prod_\lambda \K(X_\lambda)$.
\end{lemma}
\begin{proof}
For (1) define $\phi_A$ by $\phi_A (K) = K \cap A$.  For (2), let $f$ be a compact-covering map of $X$ to $Y$, and define $\phi_f (K) = f[K]$. In both cases the given map  is the desired Tukey quotient. 
For (3) define $\phi_1 : \K(\prod_\lambda X_\lambda) \to \prod_\lambda \K(X_\lambda)$ and $\phi_2  :   \prod_\lambda \K(X_\lambda) \to \K(\prod_\lambda X_\lambda)$ by $\phi_1 (K) = (\pi_\lambda (K))_\lambda$ and $\phi_2 ((K_\lambda)_\lambda) = \prod_\lambda K_\lambda$. These are the required Tukey quotients.
\end{proof}
Giving $\omega$ the discrete topology, and $\omega^\omega$ the product topology (so $\omega^\omega$ is homeomorphic to the irrationals), it follows from claim~(3) that $\K(\omega^\omega) \te \omega^\omega$. Another useful observation follows, for the proof see \cite{KM}.
\begin{lemma}\label{kkx} For every space $X$ we have $\K(X) \te \K( \K(X) )$.
\end{lemma}

A directed set $P$ is \emph{countably determined} if and only if every unbounded subset of $P$ contains a countable unbounded subset. Equivalently, $P$ is countably determined if whenever $S$ is a subset of $P$ all of whose countable subsets are bounded then $S$ is bounded (so the countable subsets of $P$ determine which subsets are bounded).
Call $P$ \emph{strongly countably determined} if it is countably determined and for every bounded subset $B$ of $P$ there is a countable subset $B_0$ such that the set of all upper bounds of $B$ and $B_0$ coincide. When $P$ is Dedekind complete the second condition is equivalent to saying that $\bigvee B = \bigvee B_0$.

\begin{lemma}
Every $\K(X)$, where $X$ is hereditarily separable, is strongly countably determined.
\end{lemma} 
\begin{proof}
Let $\K$ be a subset of $\K(X)$ such that all its countable subsets are bounded. Set $A=\bigcup \K$. Then $\K$ is bounded if  $\cl{A}$ is compact, and in this case $\cl{A}$ is the least upper bound, $\bigvee \K$, of $\K$. As $X$ is hereditarily separable, there is a countable subset $D$ of $A$ such that $\cl{D}=\cl{A}$. For each $d$ from $D$ pick a $K_d$ in $\K$ such that $d \in K_d$. Set $\K_0=\{K_d : d \in D\}$. Then $\K_0$ is a countable subset of $\K$. By assumption, it is bounded in $\K(M)$. Since  $\cl{\bigcup \K_0} = \cl{A} = \cl{\bigcup \K}$, and the first set is compact, we see that $\K$ is bounded and $\bigvee \K = \bigvee \K_0$.
\end{proof}

\begin{lemma}\label{op_bdd}
Let $P$ and $Q$ be directed sets, where $\add{P} > \omega$, and let $\phi : P \to Q$ be an order preserving map. If $Q$ is countably determined then $\phi$ has bounded image. If $Q$ is strongly countably determined then
 there is a $p_\infty$ in $P$ such that $\phi(p_\infty)$ is an upper bound of $\phi(P)$. 
\end{lemma}
\begin{proof}
Let $S=\phi(P)$. Assume $Q$ is countably determined. To show that $\phi$ has bounded image it suffices to establish that every countable subset, say $S_0$, of $S$, has an upper bound. Write $S_0=\{\phi (p) : p \in P_0\}$ where $P_0$ is a countable subset of $P$. Then as $P$ is countably additive, $P_0$ has an upper bound $p_\infty$, and, as $\phi$ is order preserving, $\phi(p_\infty)$ is an upper bound of $S_0$.

If $Q$ is strongly countably determined then there is some countable $S_0$ contained in $S$ such that every upper bound of $S_0$ is an upper bound of $S$. Then the $p_\infty$ found above is such that $\phi(p_\infty)$ is an upper bound of $S_0$, and $S$.
\end{proof}

\begin{prop}\label{PQR_to_QR}
Let $P$, $Q$ and $R$ be directed sets where $\add{P} > \omega$ and $R$ is Dedekind complete and strongly countably determined. If $P \times Q \tr R$ then $Q \tr R$.
\end{prop}
\begin{proof}
Fix $\phi : P \times Q \to R$ which is order  preserving and cofinal.
For each $q$ in $Q$, let $\phi_q : P \to R$ be $\phi_q(p) = \phi(p,q)$. By Lemma~\ref{op_bdd}, $\phi_q$ we can pick $p_q$ in $P$ such that $\phi_q(p_q)\ge \bigvee \phi_q(P)$ (indeed, $\phi_q(p_q)$ must equal $\bigvee \phi_q(P)$). 
Define $\hat{\phi} : Q \to R$ by $\hat{\phi}(q) = \phi(p_q,q) = \phi_q(p_q)$.
Let $C$ be a cofinal set in $Q$.
We verify that $\hat{\phi}(C)$ is cofinal in $R$. Take any $r$ in $R$. For some $(p,q)$ we have $\phi(p,q) \ge r$. Since $C$ is cofinal, and $\phi$ is order preserving, we can assume $q$ is in $C$. Then $\hat{\phi}(q)=\phi_q(p_q) \ge \bigvee \phi_q(P) \ge \phi_q(p) = \phi(p,q) \ge r$.
\end{proof}

\begin{lemma}\label{combine} Let $P$ and $Q$ be directed sets, with $P$  Dedekind complete. Suppose $P = \bigcup_{\alpha \in \kappa} P_\alpha$ and for each $\alpha$ we have $Q \tr (P_\alpha ,P)$.
Then $Q \times [\kappa]^{<\omega} \tr P$.
\end{lemma}
\begin{proof}
As $P$ is Dedekind complete, fix an order preserving $\phi_\alpha : Q \to P$ such that $\phi_\alpha (Q)$ is cofinal for $P_\alpha$ in $P$. Define $\phi : Q \times [\kappa]^{<\omega} \to P$ by $\phi(q,F) = \bigvee_{\alpha \in F} \phi_\alpha (q)$.
Clearly $\phi$ is  order preserving. If $p$ is any element of $P$, then $p$ is in $P_\alpha$, for some $\alpha$. Pick $q$ from $Q$ such that $\phi_\alpha (q) \ge p$. Then $\phi(q,\{\alpha\}) = \phi_\alpha (q) \ge p$, and thus $\phi$ is cofinal.
\end{proof}

\begin{lemma}\label{pow_general}
Suppose $\kappa < \mathop{add}(Q)$ and $(Q,P) \tr (Q_\alpha, P_\alpha)$ for each $\alpha  \in \kappa$. Also assume each $P_\alpha$ is Dedekind complete. Then $(Q,P) \tr (\prod_{\alpha \in \kappa}Q_\alpha, \prod_{\alpha \in \kappa}P_\alpha)$.
\end{lemma}
\begin{proof}
For each $\alpha  \in \kappa$, let $\phi_\alpha$ be an order preserving relative Tukey quotient witnessing $(Q,P) \tr (Q_\alpha, P_\alpha)$. Define $\phi(x) = \mathbf{x}$ where $\mathbf{x}(\alpha)=\phi_\alpha (x)$ for all $\alpha < \kappa$. Evidently $\phi$ is an order preserving map from $P$ to $\prod_{\alpha \in \kappa}P_\alpha$. Take any $(x_\alpha)_{\alpha < \kappa}$ in $\prod_{\alpha \in \kappa}Q_\alpha$. For $\alpha \in \kappa$, $\phi_\alpha (Q)$ is cofinal in $Q_\alpha$ and we can pick $y_\alpha \in Q$ such that $\phi_\alpha(y_\alpha) \geq x_\alpha$. Then $\{y_\alpha : \alpha < \kappa\}$ has an upper bound in $Q$, say $y$. Now we see that $\phi(y) \ge (x_\alpha)_{\alpha < \kappa}$, and thus $\phi(Q)$ is cofinal for $\prod_{\alpha \in \kappa}Q_\alpha$. So, $\phi$ is a relative Tukey quotient and $(Q,P) \tr (\prod_{\alpha \in \kappa}Q_\alpha, \prod_{\alpha \in \kappa}P_\alpha)$. 
\end{proof}

The following is well known.
\begin{lemma}\label{k}
For any cardinal $\kappa$, $[\kappa]^{<\omega} \tr P$ if and only if $\cof{P} \le \kappa$. 
\end{lemma}


\begin{lemma}\label{om}
Let $(\kappa_n)_n$ be a non decreasing sequence of cardinals with limit $\kappa$. Let $A$ be an infinite subset of $\mathbb{N}$. Then $\prod_{n \in A} [\kappa_n]^{<\omega} \te ([\kappa]^{<\omega})^\omega$.
\end{lemma}
\begin{proof} Clearly $([\kappa]^{<\omega})^\omega \tr \prod_{n \in A} [\kappa_n]^{<\omega}$. For the converse, 
since we can split the given $A$ into an infinite family of infinite subsets it suffices to show that  $\prod_{n \in A} [\kappa_n]^{<\omega} \tr [\kappa]^{<\omega}$.

Let $n_1=\min A$. Define  $\phi: \prod_{n \in A} [\kappa_n]^{<\omega} \to [\kappa]^{<\omega}$ by $\phi( (F_n)_n) = \bigcup \{ F_i : i \le |F_{n_1}|\}$. Clearly $\phi$ is well-defined and order preserving. Take any finite subset $F$ of $\kappa$. Then $F \subseteq \kappa_m$ for some $m>n_1$ in $A$. Pick $F_1 \subseteq \kappa_{n_1}$ of size $m$. Set $F_m = F$. And for $n \in A \setminus \{n_1,m\}$ set $F_n=\emptyset$. Then $(F_n)_n$ is in $\prod_{n \in A} [\kappa_n]^{<\omega}$ and $\phi((F_n)_n) = F_{n_1} \cup F_m \supseteq F$, as required for $\phi$  to have cofinal image.
\end{proof}

\begin{prop}\label{aleph_n} (1) For $n$ in $\omega$, we have $\Pk{\aleph_n} =_T \omega^\omega \times [\omega_n]^{<\omega}$, but (2) $\Pk{\aleph_\omega} >_T \omega^\omega \times [\aleph_\omega]^{<\omega}$.
\end{prop}
\begin{proof}
Clearly for any infinite $\kappa$, we have $\Pk{\kappa} \te \Pk{\kappa} \times [\kappa]^{<\omega} \tr \omom \times [\kappa]^{<\omega}$. 

To complete the proof of (1) we need to verify $\Pk{\aleph_n} \le_T \omega^\omega \times [\omega_n]^{<\omega}$. We use induction on $n$.
When $n=0$ the claim is clear. So assume $\Pk{\aleph_{n-1}} =_T \omom \times [\omega_{n-1}]^{<\omega}$, for some $n \ge 1$.
Then,  $\Pk{\aleph_n} = \bigcup_{\omega_{n-1} \le \alpha < \omega_n} \Pk{[0,\alpha]}  =_T \omom \times [\omega_{n-1}]^{<\omega} \times [\omega_n]^{<\omega} =_T \omom \times [\omega_n]^{<\omega}$, using the inductive hypothesis and Lemma~\ref{combine}.

To complete the proof of (2) we need to verify $\Pk{\aleph_\omega} \not\le_T \omega^\omega \times [\omega_\omega]^{<\omega}$. Suppose, for a contradiction, that $\phi : \omega^\omega \times [\omega_\omega]^{<\omega} \to \Pk{\aleph_\omega}$ is order preserving and has cofinal image. 

Take any $F$ in $[\omega_\omega]^{<\omega}$. Let $C_F = \bigcup_n \pi_n \left(\phi( \omega^\omega \times \{F\})\right)$. 
Then $C_F$ is countable. If not then for some $n$, the set $\pi_n \left(\phi( \omega^\omega \times \{F\})\right)$ is uncountable. Pick, for each $\alpha$ in $\omega_1$ a $\sigma_\alpha$ in $\omega^\omega$ so that $\pi_n \left(\phi(\sigma_\alpha,F)\right) \ne \pi_n \left(\phi(\sigma_\alpha',F)\right)$ whenever $\alpha \ne \alpha'$.
As  $\omega^\omega$ has calibre $(\omega_1,\omega)$, there is an infinite subset of the $\sigma_\alpha$'s with an upper bound -- but the corresponding $\pi_n \left(\phi(\sigma_\alpha,F)\right)$'s do not have an upper bound, contradicting $\phi$ order preserving.

Note that $\phi(\omega^\omega \times \{F\}) \subseteq C_F^\omega$, and so $\phi(\omega^\omega \times [\omega_\omega]^{<\omega}) \subseteq \bigcup \{C_F^\omega : F \in [\omega_\omega]^{<\omega}\} =\mathcal{C}$. We show that $\mathcal{C}$ is not cofinal.
For each $n$ in $\omega$ set $S_n= \bigcup \bigcup \{ C_F : F \in [\omega_n]^{<\omega}\}$. Then $|S_n| \le \aleph_n$. So pick $F_n \in [\omega_\omega]^{<\omega}$ disjoint from $S_n$. Since every finite subset of $\omega_\omega$ is contained in some $\omega_n$, we see that $(F_n)_n$ is not below any element of $\mathcal{C}$.
\end{proof}

\begin{lemma}\label{cof}
We have (1) $\cof{\omega^\omega \times [\omega_{\omega}]^{<\omega}}=\max{(\mathfrak{d},\aleph_\omega)}$, while  (2) $\cof{\Pk{\aleph_\omega}}=\max{(\mathfrak{d},\cof{[\aleph_\omega]^{\le \omega}})}$.
\end{lemma}
\begin{proof}
Part (1) is immediate. We show (2).
First note that clearly $\Pk{\aleph_\omega} \tq \omega^\omega$  and $\tq [\aleph_\omega]^{\le \omega}$ (for the latter consider $(F_n)_n \mapsto \bigcup_n F_n$). 
Now suppose $\mathcal{C}_0$ is cofinal in $[\aleph_\omega]^{\le \omega}$. We may assume that each element of $\mathcal{C}_0$ is infinite. For each $C$ in $\mathcal{C}_0$ enumerate $C$ as $\{ c_i : i\in \omega \}$. Fix a cofinal subset $\mathcal{D}$ of $\omega^\omega$ of size $\mathfrak{d}$, and for each $\sigma$ in $\mathcal{D}$ and $n$ set $F_n^{\sigma, C} = \{c_0, \ldots c_{\sigma (n)} \}$. Then $\mathcal{C}_1 = \{ (F_n^{\sigma, C})_n : \sigma \in \mathcal{D} \text{ and } C\in \mathcal{C}_0\}$ is cofinal in $\Pk{\aleph_\omega}$ and has size no more than $\max{(\mathfrak{d}, |\mathcal{C}_0|)}$.
\end{proof}

\section{Cofinality, Calibres and Spectrum for Total Imperfects}\label{Invs}

\subsection{Some Upper and Lower Bounds}





If $M$ is a non-locally compact metrizable space then the metric fan, $\mathbb{F}$, embeds in $M$ as a closed subspace. It is easy to see that $\K(\mathbb{F}) \te \omega^\omega$. So we immediately get:

\begin{lemma}\label{not_loc_cpt} Let $M$ be metrizable and non-locally compact. Then $\omega^\omega \leq \K(M)$.  
\end{lemma}




Recall the definition of a \emph{scattered} space. For a space $X$, let $X'$ be the set of all isolated points of $X$. Let $X^{(0)} = X$ and define $X^{(\alpha)} = X \backslash \bigcup_{\beta < \alpha} (X^{(\beta)})'$ for each $\alpha >0$. Then a space $X$ is called scattered if $X^{(\alpha)} = \emptyset$ for some ordinal $\alpha$. This $\alpha$ is called the \emph{scattered height} of $X$ and is denoted by $\h(X)$. Every countable separable metrizable compact space is scattered (with countable scattered height). 

Note that $\Q$ contains compact subsets of arbitrarily large countable scattered height (every countable ordinal embeds in $\Q$); and every non-scattered separable metrizable space contains a copy of $\Q$.  Now we present a lower bound. 

\begin{lemma}\label{omega_1_imperfect}
Let $B$ a totally imperfect separable metrizable space that is not scattered. Then $\lambda  \times \omega^\omega \leq_T \K(B)$ for every cardinal $\lambda \le \max( \omega_1,|B|)$.
\end{lemma}
\begin{proof}
Since non-scattered totally imperfect separable metrizable spaces are not locally compact we have $\omega \leq_T \omega^\omega \leq_T \K(B)$. 

First we show that $\omega_1 \leq_T \K(B)$. (And note that then $\omega_1 \times \omega^\omega \leq_T \K(B) \times \K(B) \te \K(B)$.) Fix $\alpha \in \omega_1$. Since $B$ contains $\Q$, we can pick a compact subset $K_\alpha$ of $B$ such that $\h(K_\alpha)>\alpha$. Define $\psi : \omega_1 \to \K(B)$ by $\psi(\alpha) = K_\alpha$. Consider an unbounded subset $S$ of $\omega_1$. Suppose there is $K \in \K(B)$ that bounds $\psi(S)$ from above. Since $S$ is unbounded in $\omega_1$, there exists $\alpha \in S$ with $\h(K) < \alpha < \h(K_\alpha)$. Then $K_\alpha \subseteq K$ contradicts the fact that if $X\subseteq Y$ then $\h(X) \leq \h(Y)$. Therefore, the map $\psi$ is a Tukey map.

Now suppose $\lambda$ is  no more than $|B|$. Since $\cof{\lambda} \te \lambda$, we can assume $\lambda$ is regular. The case $\lambda=\omega$ has already been dealt with, so assume $\lambda$ is uncountable. Enumerate $B=\{x_\alpha : \alpha < \kappa\}$. Define $\psi : \lambda \to \K(B)$ by $\psi (\alpha) = \{x_\alpha\}$. If $S$ is unbounded in $\lambda$ it has size $\lambda$, so $\psi(S)$ has size $\lambda$. Hence its closure in $B$ can not be compact (all compact subsets of $B$ are countable), in other words $\psi(S)$ is unbounded in $\K(B)$, as required for $\psi$ to be a Tukey map.
\end{proof}

Let $B$ be any zero dimensional, totally imperfect, separable metrizable space. Then it has a countable base, $\B =\{B_n : n\in \omega \}$, consisting of sets that are closed and open, such that $\B$ is closed under complements, finite intersections and finite unions. Note that for a compact scattered space $K$ there is $\alpha$ such that $K^{(\alpha)}$ is finite. So, if $K \in \K(B)$, then there is $\alpha \in \omega_1$ such that $K^{(\alpha)}$ is finite. For a fixed $\alpha$ and (finite) subset $F$ of $B$, let  $\K_\alpha^F (B) = \{ K\in \K(B) : K^{(\alpha)} \subseteq F, \ F\subseteq K \}$, and $\K_\alpha (B) = \bigcup \{K_\alpha^F (B) : F \subseteq B\}$. Suppose we have described elements of $\K_\beta (B)$ for each $\beta <\alpha$. Suppose $K\in \K^{\{x\}}_\alpha (B)$ for some $x\in B$. Pick a decreasing local base at $x$, $\{B'_{x,n}\}_{n\in \omega}$. Let $B_{x,0} = B \backslash B'_{x,0}$ and $B_{x,n} = B'_{x,n} \backslash B'_{x,n-1}$ for each $n\in \omega$. Then each $B_{x,n}$ is in $\B$. If we let $K_n = K\cap B_{x,n}$, we get $K=\{x\} \cup \bigcup_{n\in \omega} K_n$. Note that each $K_n$ is compact (since elements of $\B$ are closed) and is an element of  $\K_{\beta_n} (B)$ for some $\beta_n <\alpha$.

\begin{lemma}\label{ti} Let $B$ be totally imperfect separable metrizable space. Then (1) for each $\alpha$ in $\omega_1$,  $( [|B|]^{<\omega})^{\omega} \tr \K_\alpha (B)$; 
and  hence (2) $( [\max (\omega_1, |B|)^{<\omega})^{\omega} \tr \K(B)$. 
\end{lemma}
\begin{proof} Let $\kappa=\max (\omega_1, |B|)$ and $\lambda=|B|$. If $\lambda=\ctm$ then, as $|
K(B)|=|\K_\alpha (B)|=\ctm$, claims (1) and (2) are immediate from Lemma~\ref{k}. So assume $\lambda < \ctm$, and hence $B$ is zero dimensional. 
Since $\K(B)=\bigcup_{\alpha \in \omega_1} \K_\alpha (B)$,  from the first part and Lemma~\ref{combine} we get $\K(B) \leq_T ( [\lambda]^{<\omega})^{\omega} \times [\omega_1]^{<\omega} \leq_T ( [\kappa]^{<\omega})^{\omega} \times [\kappa]^{<\omega} \te ( [\kappa]^{<\omega})^{\omega}$.

We prove that $( [\lambda]^{<\omega})^{\omega} \tr \K_\alpha (B)$ by induction on $\alpha$. For the base case note that the set $\K_0 (B)$ is $[B]^{<\omega}$, so $( [\lambda]^{<\omega})^{\omega} \tr \K_0 (B)$. Now suppose $\alpha>0$. Define $\beta_n$'s for $n \in \omega$, as follows: if $\alpha$ is a limit then pick an increasing sequence, $\{\beta_n\}$, converging to $  \alpha$, otherwise let $\beta_n=\alpha-1$ for each $n$. Let $\K_{<\alpha} (B)= \bigcup_{\beta < \alpha} \K_\beta (B)= \bigcup_{n \in \omega} \K_{\beta_n} (B)$. By inductive hypothesis, for each $n$, $( [\lambda]^{<\omega})^{\omega} \tr \K_{\beta_n} (B)$. Hence by Lemma~\ref{combine}, $( [\lambda]^{<\omega})^{\omega} \tr ( [\lambda]^{<\omega})^{\omega} \times [\omega]^{<\omega} \tr \K_{<\alpha} (B)$.

Suppose that, for a fixed $x$ in $B$, we have  $( [\lambda]^{<\omega})^{\omega} \tr (\K_\alpha^{\{x\}} (B), \K_\alpha (B))$. Then for any $F \subseteq B$, we see that $( [\lambda]^{<\omega})^{\omega} \te (( [\lambda]^{<\omega})^{\omega})^{|F|} \tr \prod_{x \in F} (\K_\alpha^{\{x\}} (B),\K_\alpha (B)) \tr (K_\alpha^F (B),\K_\alpha (B))$ (for the last relation take the union). Since $\K_\alpha (B) =   \bigcup \{K_\alpha^F (B) : F \subseteq B\}$, and $( [\lambda]^{<\omega})^{\omega} \tr \K_\alpha^F (B)$,  by Lemma~\ref{combine}, we have $( [\lambda]^{<\omega})^{\omega} \te ( [\lambda]^{<\omega})^{\omega} \times [\lambda]^{<\omega} \tr ( [\lambda]^{<\omega})^{\omega} \times \left[ [B]^{<\omega} \right]^{<\omega} \tr \K_\alpha (B)$.

Fix, then, $x$ in $B$. 
Recall that associated with $x$ we have a sequence $\{ B_{x,n} \}_n$ of basic clopen sets. For each $n$, fix $\phi'_n : ( [\lambda]^{<\omega})^{\omega} \to \K_{<\alpha}(B)$ and define $\phi_n : ( [\lambda]^{<\omega})^{\omega} \to \K_{<\alpha}(B_{x,n})$ by $\phi_n(\tau) = \phi'_n(\tau) \cap B_{x,n}$. Since each $B_{x,n}$ is closed, each $\phi_n$ is a Tukey quotient. For $\sigma \in ( [\lambda]^{<\omega})^{\omega \times \omega}$ and $n\in \omega$, define $\sigma_n \in ( [\lambda]^{<\omega})^{\omega}$ by $\sigma_n(m) = \sigma(m,n)$. Now define $\phi : ( [\lambda]^{<\omega})^{\omega \times \omega} \to \K_{\alpha} (B)$ by $\phi (\sigma) = \{x\} \cup \bigcup_{n \in \omega} \phi_{n} (\sigma_n)$. Then $\phi$ is order preserving, and from our description of elements of $\K_\alpha (B)$, we see that its image is cofinal for $\K_{\alpha}^{\{x\}} (B)$ in $\K_\alpha (B)$. 
\end{proof}

With the assistance of  Proposition~\ref{aleph_n} we summarize the bounds for totally imperfect spaces as follows.

\begin{theorem}\label{TI_bounds}
Suppose $B$ is a non-scattered totally imperfect separable metrizable space and $\max (\omega_1,|B|) = \kappa$. Then

(1) $\lambda \times \omega^\omega \leq_T \K(B) \leq_T ( [\kappa]^{<\omega})^{\omega} $, for all $\lambda \le \kappa$; and

(2) if $\kappa = \aleph_n$ for some  $n\in \omega$, then $\omega_m\times \omega^\omega \leq_T \K(B) \leq_T  \omega^\omega \times [\kappa]^{<\omega}$, for all $m \le n$; 
in particular, $\omega_1 \times \omega^\omega \leq_T \K(\Q) \leq_T \omega^\omega \times [\omega_1]^{<\omega}$.
\end{theorem}

\subsection{Cofinality}

Our upper and lower Tukey bounds (Theorem~\ref{TI_bounds}, and apply Lemma~\ref{cof})  immediately  give upper and lower bounds on the cofinality of $\K(B)$ for $B$ a totally imperfect separable metrizable space. When $B$ is `small' -- has cardinality $< \aleph_\omega$ -- we can compute the cofinality of $\K(B)$ exactly in terms of the size of $B$.

\begin{cor}
Suppose $B$ is a non-scattered totally imperfect separable metrizable space and $\max (\omega_1,|B|) = \kappa$. Then

(1) $\max (\mathfrak{d}, |B|) \leq \cof{\K(B)} \leq \cof{ [\kappa]^{<\omega})^{\omega}}$,  and

(2) if $\kappa = \aleph_n$ for some  $n\in \omega$, then $\cof{\K(B)} = \max (\mathfrak{d}, |B|)$.
\end{cor}
As soon as $B$ has cardinality $\aleph_\omega$, however, the cofinality of $\K(B)$ is not determined by the size of $B$, nor determined in ZFC. 

\begin{exam} Suppose $\aleph_\omega < \ctm$. 
There are totally imperfect separable metrizable spaces $B'$ and $B''$ both of size $\aleph_\omega$ such that (1) $\cof{\K(B')} = \max (\mathfrak{d}, \aleph_\omega)$, but (2)  $\cof{\K(B'')} = \max (\mathfrak{d}, \cof{ [\aleph_\omega]^{\le \omega}})$.
Further, $\K(B'') >_T \K(B')$.
\end{exam}
\begin{proof}
By hypothesis there are, for each $n \ge 1$, subsets $B_n$ of $\R$ of cardinality $\aleph_n$. They are (necessarily, as they are uncountable but have size strictly less than $|\R|$) totally imperfect, and not scattered. We know from Theorem~\ref{TI_bounds}~(2) that  $\omega^\omega \times \aleph_n \le_T \K(B_n) \le_T   \omega^\omega \times [\aleph_n]^{<\omega}$.

Let $B' = \bigoplus_n B_n$. It has the stated properties, and since compact subsets of $B'$ are finite unions of compact subsets of finitely many $B_n$, it is straightforward to check that $\K(B') \le_T \omega^\omega \times [\aleph_\omega]^{<\omega}$. It quickly follows that $\cof{\K(B')} = \max (\mathfrak{d}, \aleph_\omega)$.

Let $B'' = B' \cup \{\ast\}$ where a basic open set around $\ast$ has the form $\{\ast\} \cup \bigoplus_{n \ge N} B_n$. Again $B''$ has the stated properties. 
Every compact subset of $B''$ is contained in a set (which is compact) of the form $\{\ast\} \cup \bigoplus_n K_n$ where each $K_n$ is compact in $B_n$. It follows that $\K(B'') \te \prod_n \K(B_n)$. Hence $\omega^\omega \times \prod_n \aleph_n \le_T \K(B'') \le_T \omega^\omega \times \prod_n [\aleph_n]^{<\omega}$. And so by Lemma~\ref{cof}, $\cof{\K(B'')} = \max (\mathfrak{d}, \cof{ [\aleph_\omega]^{\le \omega}})$.

To show that $\K(B'') \geq_T \K(B')$, pick a sequence $\{ x_n \}_n$ in $B''$ that does not converge in $B''$. For each $K \in \K(B')$ let $n_K$ be the smallest such that $K \subseteq B_{\leq n_K}$. Now define a map $\psi : \K(B') \to \K(B'')$ by $\psi(K)=K\cup \{ x_0, x_1, \ldots, x_{n_K} \}$. It can easily be verified that $\psi$ is a Tukey map.

It remains to show $\K(B') \not\tq \K(B'')$. 
For each $n \ge 1$ let $B_{\le n} = \bigoplus \{B_i : 1 \le i \le n\}$. We first show that if $\K(B_{\le n}) \tq \K(B_m)$ then $m \le n$. 
To see this note that $\omega^\omega \times [\aleph_n]^{<\omega} \tq \K(B_{\le n})$, so $\omega^\omega \times [\aleph_n]^{<\omega} \tq \K(B_m)$, and fix $\phi$ a witnessing order-preserving Tukey quotient.
For each $F$ in $[\aleph_n]^{<\omega}$ let $\K_F = \phi(\omega^\omega \times \{F\})$ and $S_F=\bigcup \K_F$. 
Then by a theorem of Christensen  \cite{Chris}, as $S_F$ has an $\omega^\omega$-ordered compact cover it is analytic. As $|B_m| <\ctm$, the analytic set $S_F$ must be countable. And so clearly $B_m = \bigcup \{ S_F : F \in [\aleph_n]^{<\omega}\}$ has size $\le \aleph_n$, and $m \le n$.

Now suppose for a contradiction that there is an order preserving Tukey quotient $\phi : \K(B') \to \K(B'')$. For each $n \ge 1$ define $\phi_n : \K(B_{\le n}) \to \K(B_{n+1})$ by $\phi_n(K) = \phi(K)  \cap B_{n+1}$ (where $B_{n+1}$ is considered as a subspace of $B''$).
Then $\phi_n$ is order preserving, but can not be a a Tukey quotient. So there is a $K_n$ in $\K(B_{n+1})$ which is not contained in any member of $\phi_n ( \K(B_{\le n}))$.

Set $K_\infty = \bigoplus \{K_n : n \ge 1\} \cup \{\ast\}$, and note this is in $\K(B'')$. Since $\K(B') =\bigcup_n \K(B_{\le n})$, by construction no element of $\phi(\K(B'))$ contains $K_\infty$. Contradicting $\phi$ a Tukey quotient.

To show that $\K(B'') \geq_T \K(B')$, pick a sequence $\{ x_n \}_n$ in $B''$ that does not converge in $B''$. For each $K \in \K(B')$ let $n_K$ be the smallest such that $K \subseteq B_{\leq n_K}$. Now define a map $\psi: \K(B') \to \K(B'')$ by $K \mapsto K\cup \{ x_0, x_1, \ldots, x_{n_K} \}$. The map $\psi$ is a Tukey map. Indeed, suppose $\mathcal{U} \subseteq \K(B')$ is unbounded. If $\bigcup \psi (\mathcal{U})$ contains infinitely many elements of $\{ x_n \}_n$, then $\psi (\mathcal{U})$ is also unbounded. On the other hand, if $\bigcup \psi (\mathcal{U})$ contains at most finitely many elements of $\{ x_n \}_n$, then, by definition of $\psi$, $\bigcup \mathcal{U}$ is contained in $B_{\leq n}$ for some $n$. Therefore, there is $m\leq n$ such that $(\bigcup \mathcal{U}) \cap B_m$ is not contained in a compact subset of $B_m$. Since $\bigcup \psi(\mathcal{U}) \subseteq \bigcup \mathcal{U}$, $\bigcup \psi(\mathcal{U})$ is unbounded. 
\end{proof}

Of course $\cof{ [\aleph_\omega]^{\le \omega}}$ is a central object of study in Shelah's {\small PCF} theory. It is known (see \cite{BM}, for example) that $\aleph_\omega < \cof{ [\aleph_\omega]^{\le \omega}} < \aleph_{\omega_4}$, and to get $\cof{ [\aleph_\omega]^{\le \omega}}$ strictly larger than $\aleph_{\omega+1}$ requires the consistency of large cardinals.
Note also that the Tukey type and cofinality of a $\K(B)$ where $B$ is totally imperfect and has size $\ge \aleph_\omega$ is dependent on what happens at $\aleph_\omega$.

\subsection{General Calibres}

A partially order set $P$ has calibre $(\kappa, \lambda, \mu)$ if for every $\kappa$-sized $P'\subseteq P$ there is a $\lambda$-sized $P''\subseteq P'$ such that every $\mu$-sized subset of $P''$ is bounded. `Calibre $(\kappa, \lambda, \lambda)$' is abbreviated to `calibre $(\kappa, \lambda)$' and `calibre $(\kappa, \kappa)$' is abbreviated to `calibre $\kappa$'. The following two straightforward facts were proven in \cite{KM} and \cite{KS}, respectively. 

\begin{lemma}\label{calibregoesdown}
If $P \tr Q$, $P$ has calibre $(\kappa,\lambda,\mu)$ and $\kappa$ is regular, then $Q$ has calibre $(\kappa,\lambda,\mu)$.
\end{lemma}

\begin{lemma}\label{caltuk}
Suppose $\kappa$ is regular. Then (1) $P$ fails to have calibre $\kappa$ if and only if $P\tr \kappa$, (2) If $P\tr [\kappa]^{< \lambda}$ then $P$ fails to have calibre $(\kappa, \lambda)$ and the converse is true if $\add{P} \geq \lambda$. 
\end{lemma}

\begin{lemma}
If $P$ is countably determined and has calibre $(\kappa, \lambda, \omega)$ then it has calibre $(\kappa,\lambda)$.
\end{lemma}
\begin{proof}
Take any $\kappa$-sized subset $S$ of $P$. By the calibre hypothesis there is a $\lambda$-sized subset $S_0$ of $S$ whose every countable subset has an upper bound. Since $P$ is countably determined we see that $S_0$ is bounded. Hence $P$ has calibre $(\kappa,\lambda,\lambda)$.
\end{proof}

\begin{cor}
If $\K(M)$, where $M$ is separable and metrizable, has calibre $(\kappa,\lambda,\omega)$ then it has caliber $(\kappa,\lambda)$.
\end{cor}

It is now well known (see \cite{KM}, for example) that $\K(M)$'s always have calibre $(\omega_1,\omega)$, and hence calibre $(\kappa, \omega)$ for all $\kappa > \omega$.
\begin{lemma}
If $M$ is separable metrizable then $\K(M)$ has calibre $(\omega_1, \omega)$
\end{lemma}
So we are left with two questions. First when does $\K(M)$ have calibre $(\kappa, \lambda)$ where $\kappa > \lambda > \omega$? This is discussed in \cite{DoMa}. The second is when does $\K(M)$ have calibre $\kappa$? Or more broadly for which $\kappa$ does $\K(M)$ have calibre $\kappa$? Since for many $\kappa$ -- for example all $\kappa > |\K(M)|$ -- $\kappa$ is vacuously a calibre, it turns out to be more convenient to investigate which $\kappa$ are \emph{not} calibres.

\subsection{Spectra}
The \emph{spectrum}, $\spec{P}$, of a directed set $P$, is the set of all infinite regular $\kappa$ which are not calibres of $\kappa$. Equivalently, $\spec{P}=\{ \kappa \ge \omega: \kappa \text{ is a regular cardinal and }  P\tr \kappa  \}$. The productivity of calibre $\kappa$  and transitivity of $\leq_T$ give the following facts about the spectrum. 
\begin{lemma}\label{prod_spec}
(1) $\spec{P_1 \times P_2} = \spec{P_1} \cup \spec{P_2}$. 
(2) If $Q\leq_T P$ then $\spec{Q} \subseteq \spec{P}$.
\end{lemma}

We write $[\lambda,\mu]^r$ for $\{\kappa : \kappa \text{ is regular and } \lambda \le \kappa \le \mu\}$. 

\begin{lemma}\label{add_cof_spec}
Let $P$ be a directed poset without the largest element. Then \\
$\add{P}$ and $\cof{\cof{P}}$ are elements of $\spec{P}$ and $\spec{P} \subseteq [\add{P}, \cof{P}]^r$.
\end{lemma}

\begin{proof}
To show that $\add{P} \leq_T P$, pick an unbounded subset of $P$ of size $\add{P}$, say $U = \{ u_\alpha : \alpha < \add{P}  \}$. Since each subsets of $P$ of size $<\add{P}$ is bounded, we can arrange for $u_\beta$ to be an upper bound  for $\{ u_\alpha : \alpha < \beta \}$ for each $\beta < \add{P}$. Moreover, since $P$ has no largest element we can ensure that $u_\beta$ is \emph{strictly} larger than each element of $\{ u_\alpha : \alpha < \beta \}$. Now $U$ is a strictly increasing unbounded set and therefore every subset of $U$ of size $\add{P}$ is also unbounded. Then the map $\psi : \add{P} \to P$ defined by $\psi (\beta) = u_\beta$ is a Tukey map.

To show that $\cof{\cof{P}} \leq_T P$ it suffices to show that $\cof{P} \leq_T P$ since $\cof{P} =_T \cof{\cof{P}}$. Let $\{p_\alpha: \alpha < \cof{P}\}$ be a cofinal subset of $P$ and define $\phi : P \to \cof{P}$ by setting $\phi(p)$ to equal some $\alpha$ such that $p_\alpha \geq p$. Then whenever $C$ is a cofinal subset of $P$, the set $\{ p_\alpha : \alpha \in \phi(C)  \}$ is also cofinal in $P$ and therefore has size at least $\cof{P}$. This implies that $\phi(C)$ has size $\cof{P} $ and therefore must be cofinal in $\cof{P}$.

Lastly, to show that $\spec{P} \subseteq [\add{P}, \cof{P}]^r$, suppose $P\tr \kappa$ and $\kappa$ is regular. Then we know that $\add{P} \leq \add{\kappa}$ and $\cof{\kappa} \leq \cof{P}$. Then the fact that $\add{\kappa} = \cof{\kappa}$ finishes the proof. 
\end{proof}

Since $\K(\omega^\omega) \te \omega^\omega$ we see $\spec{\omega^\omega} = \spec{\K(\omega^\omega)}$. So by Lemma~\ref{not_loc_cpt} the spectrum of most $\K(M)$'s contains the spectrum of $\omega^\omega$.

\begin{lemma}\label{NLC_spec} Let $M$ be separable metrizable.
If $M$ is compact then $\spec{\K(M)}=\emptyset$. If $M$ is locally compact, but not compact then $\spec{\K(M)}=\{\omega\}$.
While if $M$ is not locally compact, then $\spec{\omega^\omega} \subseteq \spec{\K(M)}$. 
\end{lemma}

Combining the upper and lower bounds given by Theorem~\ref{TI_bounds}, and the preceding lemmas  we deduce the spectrum  for some totally imperfect separable metrizable spaces. 

\begin{prop}\label{omega1inQ} Suppose $B_\kappa$ is $\kappa$-sized totally imperfect separable metrizable. Then $[\omega_1, \kappa]^r\cup \spec{\omega^\omega} \subseteq \spec{\K(B_\kappa)}$. If $\kappa = \aleph_n$ for some $n\in \omega$, then  $[\omega_1, \kappa]^r\cup \spec{\omega^\omega} = \spec{\K(B_\kappa)}$.
In particular, $\spec{\K(\Q)} = \{\omega_1\} \cup \spec{\omega^\omega}$.
\end{prop}


 In light of Lemma~\ref{NLC_spec} and Proposition~\ref{omega1inQ} we now turn to computing $\spec{\omega^\omega}$. The next lemma reduces the problem to calculating the spectrum of $(\omega^\omega,\le^*)$.
 
\begin{lemma}
(1) $\omega^\omega \tr (\omega^\omega, \leq^*)$. 
(2) $\spec{\omega^\omega} = \spec{\omega^\omega,\leq^*} \cup \{\omega\}$. 
\end{lemma}
\begin{proof}
Notice that the function defined by $f\mapsto [f]$ (the equivalence class of $f$ in $(\omega^\omega, \leq^*)$) is order preserving and cofinal. So $\omega^\omega \tr (\omega^\omega, \leq^*)$ and therefore $\spec{\omega^\omega,\leq^*} \subseteq \spec{\omega^\omega}$. 
On the other hand, suppose $\kappa \leq_T \omega^\omega$ and $\kappa$ is regular and uncountable. Let $\psi : \kappa \to \omega^\omega$ be a Tukey map. Define $\psi^* : \kappa \to (\omega^\omega,\le^*)$ by $\psi^*(\alpha)=[\psi (f)]$. Then because each equivalence class is countable while $\kappa$ is uncountable and regular, it is is easy to verify $\psi^*$ is a Tukey map.
\end{proof}

By Lemma~\ref{add_cof_spec}, we have the following corollary. 

\begin{cor} The cardinals 
$\omega, \mathfrak{b}$ and $\cof{\mathfrak{d}}$ are elements of $\spec{\omega^\omega}$, and $\spec{\omega^\omega} \subseteq \{\omega \} \cup [\mathfrak{b}, \mathfrak{d}]^r$.
\end{cor}

As $\mathfrak{b}$ is the least uncountable element of $\spec{\omega^\omega}$ we immediately have the following: 
\begin{cor}
The cardinal $\omega_1$ is in the spectrum of $\omega^\omega$ if and only if  $\omega_1 = \mathfrak{b}$.
\end{cor}

Taking $P=(\omega^\omega,\le^*)$ and $R=\omega^\omega = (\omega^\omega,\le)$ in Proposition~\ref{PQR_to_QR}:
\begin{lemma}
If $(\omega^\omega,\le^*) \times Q \tr \omega^\omega$, then $Q \tr \omega^\omega$. 
In particular, $\cof{Q} \ge \mathfrak{d}$, and $(\omega^\omega,\le^*) \times \omega \not\tr \omega^\omega$. 
\end{lemma}



Now we consider the possibilities for the spectrum of $(\omega^\omega, \leq^*)$. Recall $\mathfrak{b}$ is regular, $\mathfrak{b} \le \cof{\mathfrak{d}} \le \mathfrak{d}$, and $\mathfrak{d}$ need not be regular. We look at what can happen in the interval $[\mathfrak{b},\mathfrak{d}]^r$. As with our discussion of cofinality above, the answer is closely tied to Shelah's {\small PCF} theory. The basic definitions are as follows. As a reference see \cite{BM}, for example. 
Let $P$ be any directed set. Then $P$ has \emph{true cofinality} $\kappa$ if $P$ has a cofinal subset order isomorphic to $\kappa$; hence if  $P$ has true cofinality $\kappa$ then $P \te \kappa$. Let $S$ be a set of regular cardinals. Let $I$ be an ideal on $S$.  Recall $\pcf{S}=\{ \kappa : \exists \text{ ideal } I \text{ such that }  \kappa$ is the true cofinality of $\prod S/I\}$. Since for any ideal $I$ we evidently have $\prod S \tq \prod S / I$, we see by the definition of spectrum:

\begin{lemma}\label{pcf} For any set $S$ of regular cardinals, 
 $\pcf{S} \subseteq \spec{\prod S}$.
\end{lemma}

A set $S$ of cardinals is \emph{progressive} if $\min (S) > |S|$; and is \emph{almost progressive} if it is a finite union of progressive subsets.  
We now show -- provided $S$ is almost progressive -- that the converse is true: if $\prod S \tq \kappa$ then this can be realized in a specific fashion, namely there is an ideal $I$ on $S$ such that $\prod S/I$ has true cofinality $\kappa$, and so $\prod S \tq \prod S /I \te \kappa$.

Any countable family of uncountable cardinals is progressive.
Hence, if removing a countable subset from a set $S$ of uncountable cardinals makes it progressive, then $S$ is almost progressive. Also observe that any family $S$ of cardinals of size, $|S|$, strictly less than the first fixed point of the $\aleph$-function is almost progressive. To see this write $S= S_{<} \cup S_= \cup S_{>}$,  where $S_{<}=\{ \kappa \in S : \kappa < |S|\}$, $S_= = S \cap \{|S|\}$ and $S_{>}=\{\kappa \in S : \kappa > |S|\}$. Then $S_=$ and $S_{>}$ are clearly progressive.  By hypothesis on $S$ we have $|S_{<}| < |S|$. So we can repeat this process  on $S_{<}$. After finitely many iterations we have decomposed $S$ into progressive subsets.

For a cardinal $\kappa$, the ideal $J_{<\kappa}(S)$ is the set of subsets $S'$ of $S$ such that for every maximal ideal $\mathcal{M}$ on $S$, if $S \setminus S'$ is in $\mathcal{M}$ then $\cof{\prod S / \mathcal{M}} < \kappa$. 
We abbreviate $J_{<\kappa}(S)$ to $J_{<\kappa}$, when $S$ is clear from context.

\begin{prop}\label{prog}
Let $S$ be an almost progressive set of regular  cardinals. Then $\spec{\prod S} \subseteq \pcf{S}$.  
\end{prop}
\begin{proof} First suppose $S$ is almost progressive, and $S=S_1 \cup \cdots \cup S_n$ where each $S_i$ is progressive. Since any subset of a progressive set is progressive, we may suppose the $S_i$'s  partition $S$.  Then $\prod S = \prod S_1 \times \cdots \times \prod S_n$, 
so by Lemma~\ref{prod_spec} $\spec{\prod S} = \spec{\prod S_1} \cup \cdots \cup \spec{\prod S_n}$. Provided each $\spec{\prod S_i} \subseteq \pcf{S_i}$ (which is a subset of $\pcf{S}$) we are done. So we may assume, without loss of generality, that $S$ is progressive.

Take any regular $\kappa$ in $\spec{\prod S}$.  So there is a subset $\{f_\alpha : \alpha < \kappa\}$ of $\prod S$ and (after tidying) an order preserving surjection $\phi : \prod S \to \kappa$ so that $\phi(f_\alpha)=\alpha$. 
We show that $\{f_\alpha : \alpha < \kappa\}$ in $\prod S/J_{<\kappa}(S)$ is not bounded. 
Since (see \cite[1.1]{BM}) subsets of $\prod S/J_{<\kappa^+}$ of size $<\kappa^+$ \emph{are} bounded, we see $J_{<\kappa}$ and  $J_{<\kappa^+}$ are distinct, and so there is an $S'$ in $J_{<\kappa^+} \setminus J_{<\kappa}$. If $\mathcal{M}$ is a maximal ideal containing $S \setminus S'$ then $\prod S / \mathcal{M}$ has (true) cofinality $\kappa$, and $\kappa$ is in $\pcf{S}$, as desired.

Suppose, for a contradiction, that $\{f_\alpha : \alpha < \kappa\}$ is bounded in $\prod S/ J_{<\kappa}$ by $g$. We construct $<\!\!\kappa$-many modifications of $g$ so that every $f_\alpha$ is below one of these modifications in $\prod S$. Then $\kappa$-many $f_\alpha$'s will be bounded in $\prod S$, which is impossible.

For each $\lambda$ in $\pcf{S}$ such that $\lambda < \kappa$, pick  $G_\lambda$ which generates $J_{<\lambda^+}$ from $J_{<\lambda}$ (\cite[4.6]{BM}), and now pick $\mathcal{C}_\lambda$ of size $\lambda$ cofinal in $\prod G_\lambda$ (possible because $\max (\pcf{G_\lambda}) = \cof{\prod G_\lambda} = \lambda$, see \cite[7.10]{BM}). For any $\lambda_1, \ldots , \lambda_n$ and $h_1 \in \mathcal{C}_{\lambda_1}, \ldots , h_n \in \mathcal{C}_{\lambda_n}$ pick $g(\lambda_1, h_1, \ldots, \lambda_n, h_n)$ in $\prod S$ above $g$ and each $h_i$ (where defined). The collection $\mathcal{C}$ of all such modifications, $g(\lambda_1, h_1, \ldots, \lambda_n, h_n)$, of $g$ has size $<\kappa$.

Take any $f_\alpha$. Let $S'=\{ \mu \in S : f_\alpha (\mu) \ge g(\mu)\}$. Since $f_\alpha <_{J_{<\kappa}} g$ the set $S'$ is in $J_{<\kappa}$. So there are $G_{\lambda_1}, \ldots , G_{\lambda_n}$ whose union contains $S'$.
For $i=1, \ldots, n$ pick $h_i$ in $\mathcal{C}_{\lambda_i}$ above $f_\alpha \restriction G_{\lambda_i}$. Now $g(\lambda_1, h_1, \ldots , \lambda_n, h_n)$ is in $\mathcal{C}$ and is above $f_\alpha$ in $\prod S$ -- as required to complete the proof.
\end{proof}

\begin{theorem}\label{specomom} \ 

(1) For all subsets $S_0$ of $\spec{\omega^\omega,\le^*}$ of size strictly less than $\min ( \spec{\omega^\omega, \le^*})$ we have $\pcf{S_0} \subseteq \spec{\omega^\omega, \le^*}$.
In particular,  if $\spec{\omega^\omega, \le^*}$ is progressive  
 then $\spec{\omega^\omega, \le^*} = \pcf{\spec{\omega^\omega, \le^*}}$.
 
(2) If $S$ is an almost  progressive set of uncountable regular cardinals (respectively, and such that $S=\pcf{S}$) then there is a ccc forcing $\mathbb{P}$ such that in $\mathbb{V}^\mathbb{P}$ we have $\spec{\omega^\omega,\le^*} = \pcf{S}$ (respectively, $\spec{\omega^\omega,\le^*}=S)$.
\end{theorem}
\begin{proof} \ 

{\sl For (1):} Take any subset $S_0$ of $\spec{\omega^\omega,\le^*}$ where $|S_0| < \min ( \spec{\omega^\omega, \le^*})$. Then for each $\kappa$ in $S_0$ we have $(\omega^\omega,\le^*) \tq \kappa$. By the size restriction on $S_0$  and noting (Lemma~\ref{add_cof_spec}) that $\min ( \spec{\omega^\omega, \le^*})$ is the additivity of $(\omega^\omega,\le^*)$, we see  (Lemma~\ref{pow_general}) that  $(\omega^\omega,\le^*) \tq \prod S_0$. Hence, by Lemma~\ref{pcf}, $\pcf{S_0} \subseteq \spec{\prod S_0} \subseteq \spec{\omega^\omega, \le^*}$.

If $\spec{\omega^\omega, \le^*}$ is progressive then let $S_0=\spec{\omega^\omega, \le^*}$. Evidently, $|S_0| < \min ( \spec{\omega^\omega, \le^*})$, so by the above, $\pcf{\spec{\omega^\omega,\le^*}} = \pcf{S_0} \subseteq \spec{\omega^\omega,\le^*}$. Further, as $S_0$  is progressive, we can apply Proposition~\ref{prog} to see $\spec{\omega^\omega,\le^*} =S_0 \subseteq \spec{\prod S_0} \subseteq \pcf{S_0}$. So the claimed equality holds.

{\sl For (2):} Fix the set $S$. Let $P_S=\prod S$. Note that $P_S$ is countably directed and has no upper bound.  Hechler showed
\cite{hechler} that there is a ccc forcing notion, $\mathbb{P}$,  such that in $\mathbb{V}^\mathbb{P}$ the directed set $(\omega^\omega,\le^*)$ has a cofinal set 
order isomorphic to $P_S$, 
and so $(\omega^\omega, \le^*) \te P_S$. Then, in $\mathbb{V}^\mathbb{P}$, by 
Lemma~\ref{pcf} and 
Proposition~\ref{prog}, we 
have $\spec{\omega^\omega,\le^*}= \spec{P_S} = \pcf{S}(=S)$. 
\end{proof}

\section{Tukey Structure of $\K(M)$'s, $M$ Metrizable}\label{Struct}

\subsection{Complete Metrizability}\label{compmet}

Christensen's theorem says that if $M$ is separable metrizable (equivalently, metrizable of countable weight) and $\omega^\omega \tq \K(M)$ then $M$ is completely metrizable. Further, if $M$ and $N$ are both separable, completely metrizable, but not locally compact then $\K(M)$ and $\K(N)$ are Tukey equivalent. We show that both results are false for metrizable spaces of uncountable weight. 

Let $\kappa$ be a cardinal. The \emph{metric hedgehog} of spininess $\kappa$, denoted $H(\kappa)$, has underlying set $\{\ast\} \cup \left((0,1] \times \kappa\right)$ and metric $d(\ast, (x,\alpha))=x$, $d((x,\alpha), (x',\alpha))=|x-x'|$ and $d((x,\alpha), (x',\alpha'))=x +x'$ if $\alpha \ne \alpha'$. Every $H(\kappa)$ is completely metrizable (via the given metric) and has weight $\kappa$.
\begin{lemma}\label{hedge}
For every infinite $\kappa$ we have $\K(H(\kappa)) \te \left([\kappa]^{<\omega} \right)^\omega \te \K(H(\kappa)^\omega)$.
\end{lemma}
\begin{proof} As $H(\kappa)$ is completely metrizable, part (3) of Proposition~\ref{cm} says 
it suffices to show $\left([\kappa]^{<\omega}\right)^\omega \tq \K(H(\kappa))$. But the following $\phi$ is the required Tukey quotient: $\phi( (F_n)_n) = \{\ast\} \cup \bigcup_n \left([2^{-(n+1)}, 2^{-n}] \times F_n\right)$. 
\end{proof}

\begin{prop}\label{cm} Let $M$ be metrizable of uncountable weight $\kappa$.

(1) $\K(M) \tr [\kappa]^{<\omega}$.

(2) $\K(M) \te [\kappa]^{< \omega}$ if $M$ is locally compact.

(3) If $M$ is not locally compact then $\K(M) \tq \omega^\omega \times [\kappa]^{<\omega}$.

(4) If $M$ is completely metrizable then $\Pk{\kappa} \tq \K(M)$.

(5) If $\kappa < \aleph_\omega$, and $M$ is completely metrizable but not locally compact then $\K(M) \te \Pk{\kappa} \te \omega^\omega \times [\kappa]^{<\omega}$.
\end{prop}

\begin{proof}
We prove (1) by induction on $\kappa$. If $M$ contains a closed discrete subset $E$ of cardinality $\kappa$ ($\ast$) then $\phi : \K(M) \to [\kappa]^{<\omega}$ defined by $\phi(K)=|K \cap E|$ is order preserving and surjective. This occurs if $\kappa$ has uncountable cofinality.

So suppose $\kappa$ has countable cofinality and $(\kappa_n)_n$ is a strictly increasing sequence of regular  cardinals with limit $\kappa$. For each point $x$ in $M$ there is an open neighborhood $U_x$ of $x$ whose closure has minimal weight. Three cases arise.

Suppose, first, that there is an $x$ in $M$ such that $U_x$ has weight $\kappa$. This is equivalent to saying that $x$ has a neighborhood base, $(U_n)_n$ of sets all of weight $\kappa$. We can further assume that for all $n$ we have $\cl{U_{n+1}} \subset U_n$ and $D_n=\cl{U_n} \setminus U_{n+1}$ has weight $\kappa$. In each $D_n$ pick a closed (in both $D_n$ and $M$) discrete subset $E_n$ of size $\kappa_n$. 
Let $M_0= \{x\} \cup \bigcup_n E_n$. Then $M_0$ is a closed subset of $M$, so it suffices to show $\K(M_0) \tr [\kappa]^{<\omega}$. By Lemma~\ref{om} in fact it is sufficient to show $\K(M_0) \tr \prod_n [\kappa_n]^{<\omega}$. To see this is true define $\phi : \K(M_0) \to \prod_n [\kappa_n]^{<\omega}$ by $\phi(K) = (|K \cap E_n|)_n$. Then $\phi$ is clearly order preserving and surjective.

In the remaining cases, for every $x$ in $M$ we know that the weight of $\cl{U_x}$ is $<\kappa$. As $M$ is paracompact we can find a locally finite open cover, $\mathcal{U}$, such that if $U$ is in $\mathcal{U}$ then $w(\cl{U}) < \kappa$. 

For the second case suppose there is an $n$ such that $w(\cl{U}) \le \kappa_n$ for all $U$ in $\mathcal{U}$. For each $U$ in $\mathcal{U}$ pick $x_U$ in $U$, and set $E=\{ x_U : U \in \mathcal{U}\}$. Then $E$ is closed discrete. If $|\mathcal{U}|=|E|<\kappa$ then the weight of $M$ would be no more that $|\mathcal{U}| . \kappa_n$ which is strictly less than $\kappa$ -- a contradiction. Thus $E$ has size $\kappa$, and we are back to ($\ast$). 

In the final case we know there are, for each $n$, elements $U_n$ of $\mathcal{U}$ such that $w(\cl{U_n}) \ge \kappa_n$.  Given a compact subset $K$ of $M$, let $\mathcal{U}_K = \{ U \in \mathcal{U} : \cl{U} \cap K \ne \emptyset\}$, and note that since $\cl{\mathcal{U}}$ is locally finite, $\mathcal{U}_K$ is finite. For each $U$ in $\mathcal{U}$, by inductive hypothesis, there is a Tukey quotient $\phi_U : \K(\cl{U}) \to [w(\cl{U})]^{<\omega}$. Define $\phi : \K(M) \to [\kappa]^{<\omega}$ by $\phi(K)=\bigcup \{\phi_U (K \cap \cl{U}) : U \in \mathcal{U}_K\}$. Then $\phi$ is clearly well-defined (maps into $[\kappa]^{<\omega}$) and order preserving. Take any finite subset $F$ of $\kappa$. Then $F$ is the subset of some $\kappa_n$. Since $\phi_{U_n}$ is a Tukey quotient from $\K(\cl{U_n})$ to $[w(\cl{U_n})]^{<\omega}$, which contains $[\kappa_n]^{<\omega}$, there is a compact subset $K$ of $\cl{U_n}$  such that $\phi_{U_n}(K) \supseteq F$. And now we see $\phi(K) \supseteq F$, and $\phi$ has cofinal image.  

For (2) let $M$ be locally compact. We need to show $[\kappa]^{<\omega} \tq \K(M)$. For each $x$ in $X$ pick open $U_x$ containing $x$ with compact closure. Let $\mathcal{U}''=\{U_x : x \in X\}$, and take first a locally finite open refinement $\mathcal{U}'$ of $\mathcal{U}''$, and -- using the fact that $M$ has weight $\kappa$ -- a subcover $\mathcal{U}$ of $\mathcal{U}'$ of size no more than $\kappa$. Enumerate $\mathcal{U}'=\{ U_\alpha : \alpha < \kappa\}$. Define $\phi : [\kappa]^{<\omega} \to  \K(M)$ by $\phi(F) = \bigcup \{ \cl{U_\alpha} : \alpha \in F\}$. Then $\phi$ is clearly order preserving and has cofinal image.

Claim (3) follows immediately from part (2) and Lemma~\ref{not_loc_cpt}.
Every completely metrizable space of weight $\kappa$ embeds as a closed subspace in $H(\kappa)^\omega$. So  (4) follows from Lemma~\ref{hedge}.
Finally, (5) follows from parts (3) and (4), and Proposition~\ref{aleph_n}
\end{proof}

The next theorem says that, for a fixed uncountable cardinal $\kappa$, there is no directed set $P_\kappa$ such that if $P_\kappa \tq \K(M)$, where $M$ is metrizable of weight $\kappa$, then $M$ is completely metrizable. Indeed,
except under restrictive conditions ($\kappa <\mathfrak{d}$) and only for a limited class of spaces (locally compact), there is no directed set $P_\kappa$ such that if $\K(M) \te P_\kappa$, where $M$ is metrizable of weight $\kappa$, then $M$ is completely metrizable.

Further, although for small $\kappa$ -- those strictly less than $\aleph_\omega$ -- the non-locally compact, completely metrizable spaces of weight $\kappa$ all have the same, up to Tukey equivalence, $\K(M)$; this ceases to hold for weight $\aleph_\omega$.

\begin{theorem}
Let $M$ and $N$ be metrizable with uncountable weight $\kappa$.

(1) If $\kappa < \mathfrak{d}$ then $M$ is locally compact if and only if $\K(M) \te [\kappa]^{<\omega}$.

(2) If $M$ is either not locally compact or $\kappa \ge \mathfrak{d}$ then $\K(M) \te \K(M')$, where $M'$ is the non-completely metrizable space $M \times \Q$.

(3) If $\kappa < \aleph_\omega$ and both $M$ and $N$ are completely metrizable but not locally compact then $\K(M) \te \K(N)$.

(4) The spaces $M'=\bigoplus_n H(\aleph_n)$, and $M''=H(\aleph_\omega)$ are both completely metrizable, not locally compact, and of weight $\aleph_\omega$, but $\K(M') \not\te \K(M'')$.
\end{theorem}
\begin{proof} Note first that from Lemma~\ref{k} it is clear that for any cardinal  $\kappa$ we have $\omega^\omega \times [\kappa]^{<\omega} \te [\kappa]^{<\omega}$ if and only if $\kappa \ge \mathfrak{d}$.
Part (1) now follows from Proposition~\ref{cm}~(2) and~(3). For part (2) suppose $M$ either not locally compact or $w(M) \ge \mathfrak{d}$. Either way, from the observation above and  Proposition~\ref{cm}~(2) and~(3), we have $\K(M) \te \omega^\omega \times [w(M)]^{<\omega} \tq \omega^\omega \times [\omega_1]^{<\omega}$. By Theorem~\ref{TI_bounds} we have  $[\omega_1]^{<\omega} \tq \K(\Q)$. And so $\K(M) \tq \K(M) \times \K(\Q) \te \K(M \times \Q)$.

For (3) see Proposition~\ref{cm}~(5).
For (4) recall $M'=\bigoplus_n H(\aleph_n)$, and $M''=H(\aleph_\omega)$. We know for each $n$ that $\K(H(\aleph_n)) \te \omega^\omega \times [\aleph_n]^{<\omega}$. Since compact subsets of $M'$ are simply finite unions of compact subsets of $H(\aleph_n)$'s, it easily follows that $\K(M') \te \omega^\omega \times [\aleph_\omega]^{<\omega}$. On the other hand, $\K(M'') \te \left([\aleph_\omega]^{<\omega}\right)^\omega$. Now apply Proposition~\ref{aleph_n}~(2). 
\end{proof}

\subsection{The Position of $\K(\Q)$}
 We now continue Fremlin's analysis of the initial structure of $\K(\M)$.  
Recalling that for every co-analytic $M$ either $M$ is Polish or $\K(M) \te \K(\Q)$, the fundamental questions concern the position of $\K(\Q)$. Fremlin specifically asked in \cite{Frem} for a characterization of all metrizable $M$ such that $\K(\Q) \le_T \K(M)$.

\begin{lemma}
Let $M$ be metrizable of uncountable weight.

(1) If $M$ is not locally compact then $\K(M) \tr \K(\Q)$.

(2) If $M$ is locally compact then $\K(M) \tr \K(\Q)$ if and only if $w(M) \ge \mathfrak{d}$.
\end{lemma}
\begin{proof}
If $M$ not locally compact then $\K(M) \tq \omega^\omega \times [w(M)]^{<\omega}$ by part (2) of Proposition~\ref{cm}. But of course $\omega^\omega \times [w(M)]^{<\omega} \tq  \omega^\omega \times [\omega_1]^{<\omega}$ and $\omega^\omega \times [\omega_1]^{<\omega} \tq \K(\Q)$ by Theorem~\ref{TI_bounds}.

While if $M$ is locally compact, we have, by part (1) of Proposition~\ref{cm}, that $\K(M) \te [w(M)]^{<\omega}$. Now the second claim follows from Lemma~\ref{k} once we note that  the cofinality of $\K(\Q)$ is $\mathfrak{d}$ (which is immediate from Theorem~\ref{TI_bounds}~(2), for example).
\end{proof}
So we only need to look at separable metrizable $M$ in Fremlin's question. First we establish a general criterion relating separable metrizable spaces $M$ and $N$ when $\K(M) \tq \K(N)$.

\begin{lemma}
Let $M$ and $N$ be separable metrizable space and let $\D \subseteq \K(N)$. Then the following are equivalent:

(1) $\K(M) \tr (\D, \K(N))$, and 

(2) there is a compact metrizable space $Z$,  closed subset $D$ of $\K(M) \times Z$, and a map $f$ from $D$ to $N$ such that for each $L\in \D$, there is $K\in \K(D)$ such that $f(K) \supseteq L$.
\end{lemma}
\begin{proof}
To show that (2) implies (1), note first that, as $Z$ is compact, $\K(M) \te \K(M \times Z)$. As $D$ is a closed subset of $\K(M) \times Z$, which is a closed subset of $\K(M) \times Z)$, we see $\K(M \times Z) \tr \K(D)$. Finally the map $\phi : \K(D) \to \K(N)$ defined by $\phi (K) = f(K)$ witnesses $\K(D) \tr (\D, \K(N))$. Claim (1) follows by transitivity of $\tq$.

Suppose, then, that (1) holds, and $\phi : \K(M) \to \K(N)$ is an order preserving relative Tukey quotient. Let $Z$ be any metrizable compactification of $N$. Let $C_0 = \{ (K,L) \in \K(M) \times \K(N) : L \subseteq \phi(K)\}$. Let $C$ be the closure of $C_0$ in $\K(M) \times \K(Z)$.

We know that $C[\K(M)] \subseteq \K(N)$. Let $D=C \cap (\K(M) \times Z)$. Then $D$ is a closed subset of $\K(M) \times Z$. By the previous remark $D$ also equals $C \cap (\K(M) \times N)$. Let $f$ be the projection map from $\K(M) \times \K(Z)$ to $\K(Z)$ restricted to $D$. We verify that $f$ covers elements of $\D$.

Take any compact set $A$ in $\D$. As $\phi$ is a relative  Tukey quotient there is a $K$ in $\K(M)$ such that $\phi(K) \supseteq A$. Let $A_0 = \{ (K,\{x\}) : x \in \phi(K)\}$. Since $\{K\} \times \K(Z)$ is homeomorphic to $\K(Z)$, the set $A_0$ is a subspace of $\K(M) \times \K(Z)$ homeomorphic to $\phi(K)$, which is compact. Now we see that $A_0$ is a compact subset of $C_0$, and hence a compact subset of $C$. Also it is clear from the definitions of $D$ and $A_0$ that $A_0$ is a (compact) subset of $D$, and $f$ carries $A_0$ to $\phi(K)$ which contains $A$.
\end{proof}

Letting $\D = \K(N)$ in the previous lemma, and recalling Lemma~\ref{kkx}, we deduce the following. 

\begin{lemma}\label{compact_covering}
Let $M$ and $N$ be separable metrizable spaces. Then the following are equivalent:

(1) $\K(M) \tr \K(N)$,

(2) there is a compact metrizable space $Z$,  closed subset $D$ of $\K(M) \times Z$, and a compact covering map $f$ from $D$ onto $N$, and

(3) there is a compact metrizable space $Z$,  closed subset $D$ of $\K(M) \times Z$, and a compact covering map $f$ from $D$ onto $\K(N)$.
\end{lemma}

A space is \emph{Baire} if it satisfies the conclusion of the Baire category theorem (every countable intersection of dense open sets is dense). A space is \emph{hereditarily Baire} if every closed subspace is Baire. Clearly completely metrizable spaces are hereditarily Baire.  Recall \cite{Debs} that a  first countable space is \emph{not} hereditarily Baire if and only if it contains a closed copy of the rationals.
Also let $\K^0(X)=X$ and $\K^{n+1}(X)=\K(\K^n(X))$. By Lemma~\ref{kkx}, $\K(X) \te \K^n(X)$ for all $n \ge 1$.

\begin{theorem}\label{hB}
Let $M$ and $N$ be separable metrizable. Then the following are  equivalent:

(1) $\K(M) \tr \K(\Q)$, 

(2) $\K(M)$ is not hereditarily Baire, and

(3) for some $n$ we have $\K^n(M)$ is not hereditarily Baire.
\end{theorem}
\begin{proof} Clearly (2) implies (3). Suppose (3) holds and for some $n$ we know $\K^n(M)$ is not hereditarily Baire. Then $\K^n(M)$ contains a closed copy of the rationals, $\Q$, so $\K^{n+1}(M) = \K (\K^n(M)) \tq \K(\Q)$. But  $\K(M) \te \K^{n+1}(M)$, so (1) holds.

We show (1) implies (2). 
For a contradiction, suppose $\K(M) \tr \K(\Q)$ but $\K(M)$ is hereditarily Baire. Then by the preceding lemma there is a compact metrizable $Z$, a closed subset $D$ of $\K(M) \times Z$, and a compact-covering map $f$ of $D$ onto $\Q$.

Since $\K(M)$ is hereditarily Baire, $Z$ is compact metrizable, and $D$ is closed, we see that $D$ is hereditarily Baire. As $\Q$ is countable and $f$ is compact-covering by \cite{JW} $f$ is inductively perfect, so there is a closed subset $D'$ of $D$ such that $g=f\restriction D'$ is perfect. But then $D'$ is hereditarily Baire and $\sigma$-compact ($D'= \bigcup g^{-1} \{q\} : q \in \Q\}$), and hence Polish. In turn $\Q$, as the perfect image of $D'$ would be Polish, an obvious contradiction. 
\end{proof}

\begin{prop}
If $B$ is totally imperfect but not scattered then $\K(\Q) \le_T \K(B)$.
\end{prop}
\begin{proof}
Let $B$ be totally imperfect but not scattered. Then $B$ contains a closed subspace which is crowded (no isolated points) and totally imperfect. So we may assume $B$ is crowded and totally imperfect. Fix a compatible metric, $d$, on $B$. We show $\K(B)$ is meager, so not (hereditarily) Baire, and hence $\K(B) \tq \K(\Q)$.

For each $m$ in $\N$ set $\K_m = \{ K \in \K(B) : \exists x \in K \text{ such that } B_d(x,1/m) \cap K =\{x\}\}$. 
We verify that each $\K_m$ is nowhere dense, has closure contained in $\K_{2m}$, and their union is $\K(B)$.
The last claim is easy -- if $K$ is in $\K(B)$, then, as $B$ is totally imperfect, $K$ has an isolated point, say $x$, and so for some $m$ we have $B(x,1/m) \cap K =\{x\}$, which means $K$ is in $\K_m$.

Now fix $m$ to the end of the proof. Take any $K \notin \K_{2m}$.  By compactness, find a cover, $U_1=B(x_1^1,1/(2m)), \ldots , U_n=B(x_n^1,1/(2m))$, of $K$.
Then (as $K$ is not in $\K_{2m}$) for each $i$ pick $x_i^2 \in K \cap U_i$ distinct from $x_i^1$ such that $d(x_i^1,x_i^2) < 1/(2m)$, and open disjoint subsets $V_i^1, V_i^2$ of $U_i$ such that $x_i^1 \in V_i^1$, $x_i^2 \in V_i^2$.
Let $T$ be the set of all compact subsets of $B$ contained in $\bigcup_i U_i$ and meeting every $V_i^j$. Then $T$ is open in $\K(B)$ and contains $K$. We show $T$ is disjoint from $\K_m$. Take any $L$ in $T$. Take any $x$ in $L$. This $x$ is in some $U_i$ and for $j=1,2$ there is a $y_i^j \in V_i^j \cap L \subseteq U_i$. Then $x$ is distinct from at least one of $y_i^1$ and $y_i^2$, and $d(x,y_i^j) <1/m$ (the diameter of $U_i$). In other words, $L \notin \K_m$, as desired.

Take any $K \in \K_m$. Take any basic open neighborhood, $\langle U_1, \ldots , U_n\rangle$ of $K$. For $i=1, \ldots, n$ pick $x_i^1$ in $U_i$, and then, using crowdedness of $B$, pick $x_i^2$ in $U_i$ distinct from $x_i^1$ such that $d(x_i^1,x_i^2) < 1/m$. Let $L=\{x_1^1, x_1^2, \ldots , x_n^1, x_n^2\}$. Then $L$ is in the given basic set but is not in $\K_m$. Thus $\K_m$ is nowhere dense. 
\end{proof}

It is of interest to note that $\K(B)$'s, where $B$ is totally imperfect, are cofinal in $\K(\M)$.
\begin{prop} For every separable metrizable $M$ there is a totally imperfect, separable metrizable $B$ such that $\K(B) \tq \K(M)$.
\end{prop}
\begin{proof}
Take any separable metrizable $M$. Write $\tau_V$ for the Vietoris topology on $\K(M)$. Let $\tau_i$ be a totally imperfect, separable metrizable topology on the set $\K(M)$ (note that Bernstein sets, for example, are totally imperfect and have size $\ctm$, so $\tau_i$ exists). 
Let $\tau$ be the join of $\tau_V$ and $\tau_i$, and denote $(\K(M),\tau)$ by $B$.
Then $B$ is a separable metrizable space, whose topology is finer than $\tau_i$, and so is totally imperfect; and finer than $\tau_V$, so every compact subset of $B$ is a compact subset of $(\K(M),\tau_V)$. Recalling that if $\K$ is a compact subset of $\K(M)$ (with Vietoris topology) then $\bigcup \K$ is compact in $M$, then we see that $\phi : \K(B) \to \K(M)$ given by $\phi (K) = \bigcup K$,  is well defined, and clearly order preserving and surjective. In other words, $\K(B) \tq \K(M)$.
\end{proof}

Next we describe separable metrizable spaces $M$ such that $\K(\Q) \tr \K(M)$, and show they must either be Polish or have $\omega_1$ in the spectrum of $\K(M)$.

\begin{theorem}
For a separable metrizable space $M$, $\K(\Q)\tr \K(M)$ if and only if $M$ is a compact-covering image of a co-analytic set. 
\end{theorem}
\begin{proof}
Let $\K(\Q) \tr \K(M)$. Then, by Lemma~\ref{compact_covering}, there is a compact $Z$ and a closed $D\subseteq \K(\Q) \times Z$ such that $M$ is a compact covering image of $D$. Since $\K(\Q)$ is co-analytic, $D$ is also co-analytic. 

On the other hand, suppose $M$ is a compact covering image of co-analytic $N$. Then $\K(N) \tr \K(M)$ and, by Theorem~15 from \cite{Fr2}, $\K(\Q) \tr \K(N)$.
\end{proof}

Recall that $\omega_1$ is in the spectrum of $\K(\Q)$. However $\omega_1$ is in the spectrum of $\K(\omega^\omega)$ if and only if $\omega_1 = \mathfrak{b}$. We now see that if $\K(M) \le_T \K(\Q)$ but $M$ is not Polish then $\omega_1$ is (always) in the spectrum of $\K(M)$. 
The following lemma was proven in \cite{KM}. 

\begin{lemma}\label{key} Let $X$ compact and metrizable and let $M$ and $N$ be subspaces of $X$. 
Let $\K$ be a subset of $\K(M)$ and $\D$ be a subset of $\K(N)$. Then the following are equivalent:

(1) $(\K,\K(M)) \tr (\D,\K(N))$, and

(2) there is a closed subset $C$ of $\K(X)^2$ such that $C[\K(M)]= \bigcup \{ C([K]) : K \in \K(M)\}$ is contained in $\K(N)$ and $C[\K]\supseteq \D$.
\end{lemma}

\begin{lemma}\label{downwards_closed}
Suppose $X $ and $Y$ are separable metrizable spaces, $M\subseteq X$ and $N\subseteq Y$ and $\phi$ is a Tukey quotient witnessing $\K(M) \tr \K(N)$. Let $C_0 = \{ (K,L) : L \subseteq \phi (K) \}$ and let $C$ be a closure of $C_0$ in $\K(X)\times \K(Y)$. Then for each $K\in \K(M)$, the set $C[K] = \{ L \in \K(N) : (K,L) \in C \}$ is downwards closed, i.e. if $L \in C[K]$ and $L'\subseteq L$ then $L' \in C[K]$. Hence for each $\K \subseteq \K(M)$, $C[\K]$ is downwards closed. 
\end{lemma}
\begin{proof}
Suppose $X$, $Y$, $M$, $N$, $\phi$, $C_0$ and $C$ are as above. 
Recall that $\K(Y)$ has a base made out of open sets of the form $B(U_1, U_2, \ldots , U_n ; V_1, V_2, \ldots, V_m) = \{ L \in \K(Y) : L \subseteq \bigcup_{i\leq n} U_i \ \text{and} \ L \cap V_i \neq \emptyset \ \text{for each} \ j\leq m \}$, where each $U_i$ and $V_j$ is an open subset of $Y$. Denote $B(U_1, U_2, \ldots , U_n ; U_1, U_2, \ldots, U_n)$ by $B(U_1, U_2, \ldots , U_n)$.   

Fix $K \in \K(M)$. Let $L \in C[K]$ and suppose $L' \subseteq L$. Since $(K,L) \in C = \cl{C_0}$ and $\K(X)\times \K(Y)$ is metrizable, there exists a sequence $\{(K_n,L_n)\}_{n\in \omega}$ in $C_0$ that converges to $(K,L)$. Let $\B = \{B_n : n\in \omega  \}$ be a base of $L'$ in $\K(Y)$ such that $B_n = B(U_1^n, U_2^n, \cdots , U_{k_n}^n)$ and $\cl{\bigcup_{i\leq k_{n+1}} U_j^{n+1} } \subseteq \bigcup_{i\leq k_n} U_i^n $ for each $n\in \omega$.

Fix $m \in \omega$. Since $L' \in B_m$ we have $L \in B(Y, U_1^m, U_2^m, \cdots , U_{k_m}^m)$ and since $\{L_n\}$ converges to $L$, there exists $r_m \in \omega$  such that $L_{r_m} \in B(Y, U_1^m, U_2^m, \cdots , U_{k_m}^m)$. Set $L'_m = \bigcup_{i\leq k_m} (L_{r_m} \cap \cl{U_i^m})$, which is a compact subset of $N$. Since each $L_m' \in B_{m-1}$, $\{L'_m\}_{m\in \omega}$ converges to $L'$. Since each $L'_m \subseteq L_{r_m} \subseteq \phi(K_{r_m})$, we have that each $(K_{r_m}, L'_m) \in C_0$ and $\{(K_{r_m}, L'_m)\}_{m\in \omega}$ converges to $(K,L')$. Therefore $(K,L') \in C$ and $L' \in C[K]$. 
\end{proof}

\begin{prop}\label{om1inspec}
If $M$ is a separable metrizable space and $\K(\Q) \tr \K(M)$ then either $M$ is Polish or $\K(M) \tr \omega_1$.
\end{prop}
\begin{proof}
Let $C$ be the closed set given by the preceding result that witnesses $\K(\Q) \tr \K(M)$. Let $\mathcal{L}_\alpha = C[\K_\alpha (\Q)]$ for each $\alpha \in \omega_1$. Then $\{ \mathcal{L}_\alpha : \alpha \in \omega_1 \}$ is an increasing sequence of downwards closed collections whose union is $\K(M)$. 
From Lemma~\ref{ti}(1) (with $B=\Q$) and Lemma~\ref{key}, $\omega^\omega \tr \K_\alpha (\Q) \tr (\K_\alpha (\Q), \K(\Q))\tr (\mathcal{L}_\alpha, \K(M))$ for each $\alpha \in \omega_1$. If there exists $\alpha \in \omega_1$ such that $\K(M) = \mathcal{L}_\alpha$ then we have $\omega^\omega \tr \K(M)$, which means (by Christensen's theorem) that $M$ must be Polish. 

If no such $\alpha$ exists, we may assume, without loss of generality, that $\{ \mathcal{L}_\alpha : \alpha \in \omega_1 \}$ is a strictly increasing sequence. Define $\phi : \K(M) \to \omega_1$ by setting $\phi (L) = \alpha$, where $\alpha$ is the smallest ordinal such that $L \in \mathcal{L}_\alpha$. Since each $\mathcal{L}_\alpha$ is downwards closed, $\phi$ is order preserving.
 Since $\{ \mathcal{L}_\alpha : \alpha \in \omega_1 \}$ is a strictly increasing sequence, $\phi (\K(M))$ is cofinal in $\omega_1$. Thus $\K(M) \tr \omega_1$, as claimed.
\end{proof}

\subsection{The Position of $\K(M)$, for $M$ Analytic}

We now examine the relationship between $\K(M)$'s where $M$ is analytic but not co-analytic (equivalently, not Borel) and $\K(\Q)$. There are four possibilities: (i) $\K(\Q) <_T \K(M)$, (ii) $\K(M)$ and $\K(\Q)$ are Tukey incomparable, (iii) $\K(M) <_T \K(\Q)$, or (iv) $\K(M) \te \K(\Q)$.
Fremlin showed, in ZFC, that there are analytic $M$ such that (i) holds. He also showed that under Projective Determinacy for \emph{all} analytic, non-Borel $M$ we have $\K(\Q) <_T \K(M)$. In fact, `for all analytic, non-Polish $M$ we have $\K(\Q) \le_T \K(M)$' holds under much weaker hypotheses. To see this recall that Kanovei and Ostrand showed that the statements, `every hereditarily Baire analytic set is Polish' and `there are no totally imperfect, uncountable co-analytic sets', are equivalent, and are consistent and independent. If in some model every hereditarily Baire analytic set is Polish and for some analytic $M$ we have $\K(\Q) \not\le_T \K(M)$, then by Theorem~\ref{hB} $M$ is Polish.

That there are analytic, non-Borel $M$'s satisfying (iii) and (iv) under $\mathbb{V}=\mathbb{L}$ is established in \cite{GaMedZdom}. We now show that there are analytic, non-Borel $M$'s satisfying (ii).

\begin{exam}[$\omega_1 < \mathfrak{p}$ and $\omega_1^L = \omega_1$]
There is an analytic, non-Borel set $M$ such that $\K(M) \not\tq \K(\Q)$ and $\K(\Q) \not\tq \K(M)$.

Further, $\K(M)$ is hereditarily Baire and has calibre $\omega_1$.
\end{exam}
\begin{proof}
Let $N$ be an $\omega_1$-sized subset of the Cantor set. Let $M$ be the complement of $N$ in the Canor set. As Fremlin has pointed out, under the given set theoretic assumptions, $M$ is analytic, not Borel, and $\K(M) \not\tq \omega_1$. The latter is the same as saying that $\K(M)$ is calibre $\omega_1$. 

As $\K(\Q)$ is not calibre $\omega_1$, we have $\K(M) \not\tq \K(\Q)$. Hence $\K(M)$ is hereditarily Baire (Theorem~\ref{hB}). By Proposition~\ref{om1inspec} we also see $\K(\Q) \not\tq \K(M)$.
\end{proof}

\end{document}